\documentclass[10pt, journal,twocolumn]{IEEEtran}
\IEEEoverridecommandlockouts

\usepackage{graphics} 
\usepackage[pdftex]{graphicx}
\usepackage{epstopdf}
\usepackage[cmex10]{amsmath}
\usepackage{amssymb}  
\usepackage{enumerate}
\usepackage{setspace}

\newtheorem{example}{Example}
\newtheorem{theorem}{Theorem}
\newtheorem{lemma}{Lemma}
\newtheorem{remark}{Remark}

\newtheorem{conjecture}{Conjecture}
\newtheorem{proof}{Proof}

\def\R{\mathbb{R}}

\def\t{\tau}
\def\wto{\rightharpoonup}

\title{\LARGE \bf Extremum Seeking for Stabilization of Systems Not Affine in Control}

\author{Alexander~Scheinker and David~Scheinker%
\thanks{{Alexander~Scheinker is a staff member at Los Alamos National Laboratory with the RF Control Group, Los Alamos National Laboratory, Los Alamos, NM 87545, USA, {\tt\small ascheink@lanl.gov}.}}%
\thanks{{David Scheinker is a lecturer at The Stanford University Department of Management Science and Engineering, Stanford, CA 94305, USA, {\tt\small dscheink@stanford.edu}.}}%
\thanks{{This research was supported by Los Alamos National Laboratory.}}%
}

\begin{document}

\maketitle
\thispagestyle{empty}



\bigskip
\begin{abstract}
In \cite{ref-Sch-Krstic-TAC} a form of extremum seeking for control (ESC) was developed for the stabilization of uncertain nonlinear systems. In ESC the extremum seeker controls the systems through feedback rather than fine tuning a controller. 
The ESC results, and other related results, apply only to systems affine in control. 
However, in most physical systems the control effort enters the system's dynamics through a nonlinear function, such as an input with deadzone and saturation. 
In this work, we utilize our previous results on ESC to develop stabilizing controllers for systems of practical interest that are non-affine in control. 
\end{abstract}


\section{Introduction}


\paragraph*{Motivation} 
Extremum seeking optimization (ES) has been a model-independent approach to optimization since the 1920s \cite{ref-lebES}. Studies in the 50s \cite{ref-ESOst}-\cite{ref-ESVol} paved the way for stability results \cite{ref-krstic,ref-tan06}. Recent developments include the application of ES to systems with uncertainties \cite{ref-guay1}, the control of a tunable thermoacoustic cooler \cite{ref-ES-Paek}, a non-gradient approach to global ES \cite{ref-ES-Manzie}, stabilization of a nonlinear system with parametric uncertainties when a system model is available \cite{ref-guay1}, power optimization of photovoltaic micro-converters \cite{ref-Krstic-power}, stochastic ES \cite{ref-Krstic-stoch}, discrete-time systems \cite{ref-guay2}, and has been applied to moderately unstable systems \cite{ref-Krstic-unst}. Recently, Lie bracket-based averaging results of Kurzweil, Jarnik, Sussmann, and Liu \cite{ref-kurz-jar-1,ref-suss-liu}, \cite{ref-mor-aeyel-00} have been applied to ES \cite{ref-durr-stan-john-11}. A broad review of developments is summarized in \cite{ref-ES-REV}.

In \cite{ref-Sch-Krstic-TAC} Extremum Seeking for Control (ESC) was introduced, to stabilize unknown, open-loop unstable systems by using the extremum seeker itself as a high frequency feedback control, an approach related to vibrational control \cite{ref-kapitza,ref-Meerk}. 
ESC has been applied for the optimization of high voltage converter modulators \cite{ref-HVCM}, inverted pendulum stabilization \cite{ref-pend}, has been studied on manifolds \cite{ref-manifold}. Non-smooth ESC, in which oscillation of the control effort disappears as equilibrium is approached, has been developed \cite{ref-Sch-Krstic-nonC2}. 

In \cite{ref-Sch-NewES}, a new, bounded form of ESC was developed in which the control efforts and parameter update rates have analytically known bounds. The bounded ESC approach is especially useful for digital implementation and for extremely noisy systems, was recently demonstrated in hardware \cite{ref-Sch-NIM}, and was recently extended from smoothly varying sinusoidal-like functions to a much larger class of functions \cite{ref-Sch-Sch}. 

All of the methods of ESC for stabilization described above assume the unknown systems are affine in control, 
of the form:
\begin{equation}
	\dot{x} = f(x,t) + g(x,t)u.
\end{equation}
However, in most physical systems the control effort enters the system's dynamics through a nonlinear function, 
such as an input with deadzone and saturation, see for example \cite{ref-nonaff0} for an overview and control approaches. 
Thus, it is a major limitation that ESC applies only to systems affine in control.

\paragraph*{Results of the paper}
In this paper we study ESC for vector-valued systems not affine in control, of the form:
\begin{flalign}
	& \dot{x} = f(x,t) + g(x,t,u), \ \ g(x,t,u) = \sum_{i=0}^{m}g_i(x,t)u^{2i+1}, \label{odd} \\
	& y = \psi(x,t),
\end{flalign}
where $g$ is a control non-linearity given as an odd polynomial in $u$. 
In the case of full state measurements, $y=x$, Theorem \ref{thm1} reduces controlling system (\ref{odd}) to 
the significantly easier problem of controlling the averaged system
\begin{equation}
	\dot{\bar{x}} = f(\bar{x},t) -k\alpha K_{g}g_m(\bar{x},t)g^T_m(\bar{x},t)\left ( \nabla_x V \right )^T, \label{AvgSys}
\end{equation}
where $\nabla_x V(x,t) = \left  ( \frac{\partial V}{\partial x_1}, \dots, \frac{\partial V}{\partial x_n} \right )$ and $K_g$ is a constant which depends on $g$. If there exist $k$, $\alpha$, and $V$ that stabilize system (\ref{AvgSys}), then there exists $\omega^\star$ such that for all $\omega>\omega^\star$ the following will stabilize system (\ref{odd})
\begin{equation}
	u(x,t) = \left(\alpha\omega\right)^{\frac{1}{2(2m+1)}}\cos \left ( \omega t + k V(x,t) \right ).
\end{equation} 

\paragraph{Application to classical ES problems}

If, as in standard extremum seeking, the goal is optimization of a system with analytically unknown output function, $y = \psi(x,t)$, we utilize the controller
\begin{equation}
	u(x,t) = \left ( \alpha\omega \right )^{\frac{1}{2(2m+1)}}\cos(\omega t + ky),
\end{equation}
which results in the average system
\begin{equation}
	\dot{\bar{x}} = f(\bar{x},t) - k\alpha K_g g_m(\bar{x},t)g^T(\bar{x},t)\left ( \nabla_{\bar{x}} \psi(\bar{x},t) \right )^T,
\end{equation}
which performs a gradient descent of the unknown function $\psi$ when $k\alpha>0$ is chosen sufficiently large.

%

\section{Background}\label{sec:Back}

%

In this section we introduce the notions of GUAS and SPUAS stability \cite{ref-mor-aeyel-00} and Theorem \ref{thm:Sch2} from \cite{ref-Sch-Sch}, which we use to prove our main result. 
Consider two systems, $\dot{x}=f(t,x)$ and $\dot{\tilde{x}}=f^\epsilon(t,x)$, and their trajectories which pass through the point $x_0$ at time $t_0$, $\psi(t,t_0,x_0)$ and $\psi^\epsilon(t,t_0,x_0)$. 
The results of \cite{ref-mor-aeyel-00} show that if on a given compact set $K$, for any given $\delta>0$, and any $T>0$, there exists $\epsilon^\star$ such that for all $\epsilon < \epsilon^\star$, $\sup_{t\in[t_0,t_0+T]}\left | \psi(t,t_0,x_0) - \psi^\epsilon(t,t_0,x_0) \right | < \delta$, then if the origin is a globally uniformly asymptotically stable (GUAS) equilibrium point of $\dot{x}=f(t,x)$, then the origin is also an $\epsilon$ - semi globally practically uniformly asymptotically stable (SPUAS) equilibrium point of $\dot{\tilde{x}}=f^\epsilon(t,x)$. 

The two central ingredients of Theorem \ref{thm:Sch2} are the Arzel$\mathrm{\grave{a}}$-Ascoli Theorem 
and the notion of weak convergence (see, e.g. \cite{Rudin}). 
A sequence of functions $\{f_i\} \subset L^2[0,1]$ is said to converge weakly to $f$ in $L^2[0,1]$, denoted $f_i\wto f$, if 
\begin{equation} 
	\lim_{i \to \infty} \int_0^1 f_i(\t) g(\t) d\t = \int_0^1 f(\t) g(\t) d\t, \quad \forall g \in L^2[0,1]. \nonumber
\end{equation}
For example, when $\omega_1 \neq \omega_2$, Basic Hilbert Space theory gives, for $\omega_i \rightarrow \infty$: $\cos\left (\omega_i  t \right )\sin\left (\omega_j  t \right )  \wto  0$, $\cos^2\left (\omega_i  t \right ), \sin^2\left (\omega_i  t \right ) \wto  \frac{1}{2}$, $\cos\left (\omega_1  t \right )\cos\left (\omega_2  t \right )  \wto  0$, $\sin \left (\omega_1  t \right ) \sin\left (\omega_2  t \right ) \wto  0$.

In what follows, we use the notation $u(y,t)=u(\psi(x,t),t)$ to emphasize that the controller $u$ is a function of $t$ and of, a potentially unknown, function $\psi(x,t)$, i.e. that $u(y,t)$ need not have direct access to $x$. 
\begin{theorem}\label{thm:Sch2} \cite{ref-Sch-Sch}
Consider the vector-valued system
\begin{eqnarray}
	\dot{x} &=& f(x,t)+g(x,t)u(y,t), \ y = \psi(x,t),  \label{x1sys}
\end{eqnarray}
where $x\in \mathbb{R}^n$, and the functions $f : \mathbb{R}^n \times \mathbb{R} \rightarrow \mathbb{R}^n$, $g : \mathbb{R}^n \times \mathbb{R} \rightarrow \mathbb{R}^{n \times n}$, $\psi : \mathbb{R}^n \times \mathbb{R} \rightarrow \mathbb{R}$ are unknown. 
Assume that $f$ and $g$ are twice continuously differentiable with respect to $x$ and assume that the value $y$ of $\psi(x,t)$ is available for measurement. 
Consider a controller 
$u$ given by
\begin{equation}
	u(y,t) = \sum_{i=1}^{m}k_{i}(y,t)h_{i,\omega}(t), \quad k_i: \mathbb{R} \times \mathbb{R} \rightarrow \mathbb{R}^n, \label{x1sysc}
\end{equation}
where the functions $k_i(y,t)$ are continuously differentiable and the scalar functions $\quad h_{i,\omega}(t)$ are piece-wise continuous. System (\ref{x1sys}), (\ref{x1sysc}) has the following equivalent closed-loop form 
\begin{eqnarray}
	\dot{x}(t) &=& f(x,t) + \sum_{i=1}^{m} b_{i}(x,t)h_{i,\omega}(t), \label{x1syscl} \\
	b_{i}(x,t) &=& g(x,t)k_{i}\left ( \psi(x,t),t \right ). 
\end{eqnarray}
Suppose that there exist functions (which will often be constant) $\{\lambda_{i,j}\}$ such that the 
functions defined as $ H_{i,\omega}(t) = \int_{t_0}^{t}h_{i,\omega}(\tau)d\tau$ 
satisfy the uniform limits and weak limits for all $t\in[t_0,t_0+T]$
\begin{equation}
	\lim_{\omega \rightarrow \infty} H_{i,\omega}(t) = 0, h_{i,\omega}(t) H_{j,\omega}(t) \wto  \lambda_{i,j}(t) \label{weakLimits},
\end{equation}
Consider also the average system related to (\ref{x1syscl}) as follows
\begin{equation}
	\dot{\bar{x}} = f(\bar{x},t) - \sum_{i\neq j=1}^{m} \lambda_{i,j}(t)  \frac{ \partial b_i(\bar{x},t) }{\partial \bar{x}}  b_j(\bar{x},t), \quad \bar{x}(0)=x(0). \label{x1sysave}
\end{equation}
For any compact set $K \subset \mathbb{R}^n$, any $t_0, T \in \mathbb{R}_{\geq 0}$, and any $\delta>0$, there exists $\omega^\star$ such that for each $\omega>\omega^\star$, the trajectories of (\ref{x1syscl}) and (\ref{x1sysave}), satisfy
$\left \| x(t) - \bar{x}(t) \right \| < \delta$ for all $t\in [ t_0, t_0 + T ]$.
Therefore, by \cite{ref-mor-aeyel-00}, uniform asymptotic stability of (\ref{x1sysave}) over $K$ implies that (\ref{x1syscl}) is 
$\frac{1}{\omega}$-SPUAS. 
\end{theorem}

%

\section{The Main Result}\label{sec:NonAffine}

%
\begin{theorem}\label{thm1} Fix the system 
\begin{equation}
	\dot{x} = f(x,t)+\sum_{n=0}^{m}g_n(x,t)u^{2n+1}(x,t) \label{NonLinearSysOdd}, 
\end{equation}	
\begin{equation}
	f : \mathbb{R}^n \times \mathbb{R} \rightarrow \mathbb{R}^n, \ g_i : \mathbb{R}^n \times \mathbb{R} \rightarrow \mathbb{R}^{n \times n},  \nonumber
\end{equation}
where the functions $f$ and $g_i$ are twice continuously differentiable with respect to $x$ and piecewise differentiable with respect to $t$. 
Consider the related averaged system
\begin{eqnarray}
	\dot{\bar{x}} &=& f(\bar{x},t) - k\alpha A_m \left ( \frac{g_m(\bar{x},t)g^T_m(\bar{x},t)}{2^{4m+1}} \right )\left ( \nabla_{\bar{x}} V \right )^T, \nonumber \\
	A_m &=& \sum_{l=0}^{m}\binom{2m+1}{l}^2 \label{nonafave}
\end{eqnarray}
with $\bar{x}(0) = x(0)$, where $V : \mathbb{R}^n \times \mathbb{R} \rightarrow \mathbb{R}$ is twice continuously differentiable with respect to $x$ and piecewise differentiable with respect to $t$.  If there exist $k$, $\alpha$, such that $x^\star$ is a GUAS equilibrium point of (\ref{nonafave}), then (\ref{nonafave}) remains in some compact set $K$ for all $t>0$. In this case, for any $\delta>0$ there exists an $\omega^\star$ such that for all $\omega>\omega^\star$, using the controller
\begin{eqnarray}
	u_\omega(x,t) &=& \left(\alpha\omega\right)^{\frac{1}{2(2m+1)}}\cos \left ( \omega t + k V(x,t) \right ), \label{nonaf2}
\end{eqnarray} 
ensures that $x(t)$ remains within $\delta$ of $\bar{x}(t)$ for all $t>0$, and therefore $x(t)$ remains within the compact set $K_{+\delta} = \left \{ y \in \mathbb{R}^n \ \mathrm{s.t.} \ \sup_{x\in K} \left | x-y \right | \leq \delta \right \}$. Thus, $x^\star$ is a $\frac{1}{\omega}$ -SPUAS equilibrium point of (\ref{NonLinearSysOdd}), (\ref{nonaf2}). 

\end{theorem}
\begin{proof} The closed-loop form of system (\ref{NonLinearSysOdd}), (\ref{nonaf2}) is (throughout the proof we omit the arguments of $f$, $g_i$, and $V$ to simplify the notation)
\begin{equation}
	\dot{x} = f(x,t) + \sum_{n=0}^{m}g_n(x,t)\left(\alpha\omega\right)^{\frac{2n+1}{2(2m+1)}}\cos^{2n+1} \left ( \omega t + k V \right ).
\end{equation}
Let $b_n = 2n+1$ and $b_{n,l} = 2n+1-2l$, apply trigonometric power identities, and rewrite the sum as
\begin{equation} 
\sum_{n=0}^{m}g_n\sum_{l=0}^{n}\frac{\left(\alpha\omega\right)^{\frac{b_n}{2b_m}}}{2^{2n}}\binom{b_n}{l}\cos \bigg ( b_{n,l}\left ( \omega t + k V \right ) \bigg ). \label{s}
\end{equation} 
Apply trigonometric identities to expand (\ref{s}) as the sum of
\begin{equation}
	\sum_{n=0}^{m}g_n\sum_{l=0}^{n}\underbrace{\frac{\left(\alpha\omega\right)^{\frac{b_n}{2b_m}}}{2^{2n}}\binom{b_n}{l}\cos \bigg ( b_{n,l}\omega t \bigg )}_{h_{c,n,l,\omega}(t)}\cos \bigg ( b_{n,l}kV \bigg ) \nonumber
\end{equation}
and
\begin{equation}
	-\sum_{n=0}^{m}g_n \sum_{l=0}^{n}\underbrace{ \frac{\left(\alpha\omega\right)^{\frac{b_n}{2b_m}}}{2^{2n}}\binom{b_n}{l}\sin \bigg ( b_{n,l}\omega t \bigg )}_{h_{s,n,l,\omega}(t)}\sin \bigg ( b_{n,l}kV \bigg ). \nonumber
\end{equation}
For all $n\leq m$ and $\omega \geq \frac{1}{\alpha}$, $\left(\alpha\omega\right)^{\frac{2n+1}{2(2m+1)}} \leq \sqrt{\alpha \omega}$, 
the functions $h_{s,n,l,\omega}(t)$ and $h_{c,n,l,\omega}(t)$ have uniform limits
\begin{equation}
	\lim_{\omega \rightarrow \infty} H_{c/s,n,l,\omega}(t) = \lim_{\omega \rightarrow \infty} \int_{t_0}^{t}h_{c/s,n,l,\omega}(\tau)d\tau = 0, 
\end{equation}
and for all $n_1,n_2<m$, have weak limits
\begin{equation}
	h_{c/s,n_1,i,\omega}(t) H_{c/s,n_2,j,\omega}(t) \wto  0.
\end{equation}
For $n=m$, we must consider all of the terms
\begin{equation}
	h_{c,m,l,\omega}(t) = \frac{1}{2^{2m}} \binom{b_m}{l}\sqrt{\alpha \omega} \cos \left ( b_{m,l} \omega t \right),
\end{equation}
\begin{equation}
	h_{s,m,l,\omega}(t) = \frac{1}{2^{2m}} \binom{b_m}{l}\sqrt{\alpha \omega} \sin \left ( b_{m,l} \omega t \right).
\end{equation}
The products $h_{c,m,i,\omega}(t) H_{s,m,j,\omega}(t)$ are given by
\begin{eqnarray}
	&& -\alpha \left ( \frac{1}{2^{2m}} \right )^2 \binom{b_m}{l}^2\cos^2 \left ( b_{m,l} \omega t \right) \nonumber \\
	&& +\alpha \left ( \frac{1}{2^{2m}} \right )^2 \binom{b_m}{l}^2 \cos \left ( b_{m,l} \omega t_0 \right)\cos \left ( b_{m,l} \omega t \right).
\end{eqnarray}
The terms $ \cos^2 \left ( b_{m,l} \omega t \right)$ and $\cos\left ( b_{m,l} \omega t \right)$ weakly converge to $1/2$ and $0$, respectively, therefore
\begin{equation}
	h_{c,m,i,\omega}(t) H_{s,m,j,\omega}(t)  \wto - \alpha\frac{1}{2^{4m+1}}\binom{b_m}{l}^2,
\end{equation}
\begin{equation}
	h_{s,m,i,\omega}(t) H_{c,m,j,\omega}(t)  \wto \underbrace{\alpha\frac{1}{2^{4m+1}}\binom{b_m}{l}^2}_{a_{m,l}}. \label{aml}
\end{equation}
Therefore, by application of Theorem \ref{thm:Sch2}, the $b_{m,l} \omega$ frequency component's contribution to the average dynamics are
\begin{flalign} 
-\alpha a_{m,l}g_m\cos(kb_{m,l}V)\frac{\partial}{\partial \bar{x}}\bigg ( g_m \sin(kb_{m,l}V) \bigg ) \label{lm1} \\
+\alpha a_{m,l}g_m\sin(kb_{m,l}V)\frac{\partial}{\partial \bar{x}}\bigg ( g_m \cos(kb_{m,l}V) \bigg ). \label{lm2}
\end{flalign} 
The term (\ref{lm1}) can be expanded as
\begin{flalign} 
-\alpha a_{m,l}g_m\cos b_{m,l}V) \frac{\partial g_m}{\partial \bar{x}} \sin(kb_{m,l}V) \nonumber \\
-\alpha a_{m,l}g_m\cos(kb_{m,l}V)kb_{m,l}\left ( \nabla_{\bar{x}} V \right )^Tg_m \cos(kb_{m,l}V), \label{lm12}
\end{flalign} 
and (\ref{lm2}) can be expanded as
\begin{flalign} 
+\alpha a_{m,l}g_m\sin(kb_{m,l}V) \frac{\partial g_m}{\partial \bar{x}} \cos(kb_{m,l}V) \nonumber \\
-\alpha a_{m,l}g_m\sin(kb_{m,l}V)kb_{m,l}\left ( \nabla_{\bar{x}} V \right )^Tg_m \sin(kb_{m,l}V). \label{lm22}
\end{flalign}
Adding (\ref{lm12}) and (\ref{lm22}) we are left with
\begin{equation}
	-k\alpha b_{m,l}a_{m,l}g_mg^T_m\left ( \nabla_{\bar{x}} V \right )^T
\end{equation}
Plugging in for the values of $a_{m,l}$ from (\ref{aml}), the overall average system dynamics are given by
\begin{equation}
	\dot{\bar{x}} = f(\bar{x},t) - k\alpha \sum_{l=0}^{m}\binom{b_m}{l}^2 \left ( \frac{g_m(\bar{x},t)g^T_m(\bar{x},t)}{2^{4m+1}} \right )\left ( \nabla_{\bar{x}} V \right )^T \nonumber
\end{equation}
and the desired convergence and the existence of the appropriate $\omega$ are guaranteed by Theorem \ref{thm:Sch2}. 
\end{proof}

Theorem \ref{thm1} implies that to stabilize system (\ref{NonLinearSysOdd}) one must choose 
the gain $k$, dithering amplitude $\alpha$, and a function $V(x,t)$ relative to upper and lower bounds on 
$\left \| f(x,t) \right \|$ and $\left \| g_m(x,t)g^T_m(x,t) \right \|$, respectively. Once $k$, $\alpha$, and $V$ are chosen, there exists a sufficiently large $\omega^\star$, such that for all $\omega>\omega^\star$, the above results hold. In practice we specify $\omega$ and design the controller to maintain the values of $u$ within some compact set $K_u$. Because $g_m(x,t)g^T_m(x,t) \geq 0$ one need not know the sign of $g_m(x,t)$ which simplifies significantly the stabilization problem. 
%

\section{An application of the Main Result}\label{sec:ExNonAffine}

%

In this section we give sufficient conditions to use our main result to control a general nonlinear system of the form:
\begin{equation}
	\dot{x} =  f(x,t) + h(x,t), \  h(x,t) = \sum_{n=0}^{m}g_n(x,t)v_n(u(x,t)) \label{NonLinearSys}
\end{equation} 
where the functions $f$ and $g_n$ are twice differentiable and the functions $v_n: \mathbb{R} \rightarrow \mathbb{R}$ 
are continuous, odd functions of $u$, i.e., 
for fixed $x$ and $t$, $v(-u(x,t))=-v(u(x,t))$.  We note that the non-affine system (\ref{NonLinearSys}) does not satisfy the hypotheses of the standard Lie Bracket averaging results of 
\cite{ref-durr-stan-john-11,ref-Sch-Krstic-TAC,ref-Sch-NewES}. 

Lemma \ref{Approx}, a technical result given below, shows that system (\ref{NonLinearSys}) may be written as 
$\dot{x} = \dot{\tilde{x}} + E$, 
an arbitrarily small perturbation of a system of the form of Theorem \ref{thm1}:
\begin{equation}
\dot{\tilde{x}}=f(\tilde{x},t) + p(\tilde{x},t), \ p = \sum_{n=0}^{m}\tilde{g}_n(\tilde{x},t) u^{2n+1}(\tilde{x},t). \label{nonaf1}
\end{equation}
If we stabilize the origin of the odd polynomial system (\ref{nonaf1}) with an appropriate Lyapunov function, then the results of Chatper 9 
of \cite{ref-khalil} imply that the origin of system (\ref{NonLinearSys}) is also stable for a set of initial conditions which can be made arbitrarily large by making the pertubation arbitrarily small.

\begin{lemma}\label{Approx}
Fix the system (\ref{NonLinearSys}), $[t_0, t_0+T]\subset\R_{\geq0}$, $K\subset \R^n$, $K_u \subset \R$ and $\epsilon>0$. 
There exists a perturbation $E$ such that $\dot{x} = \dot{\tilde{x}} + E$ and $||E||_{K\times [t_0,t_0+T]} < \epsilon$.
\end{lemma}

\begin{proof}
Since each $g_n$ is continuous and $K\times [t_0,t_0+T]$ is compact there exists a $G > ||g_n||_{K\times [t_0,t_0+T]}$ for all $n$. The M\"{u}ntz-Sz\'{a}sz Theorem, a generalization of the Stone-Weierstrass Theorem (see, e.g. \cite{Rudin}), 
implies that the odd polynomials are dense in the set of odd continuous functions on $[t_0,t_0+T]$. 
Thus, for each $n$ there exists an odd polynomial $p_n$ so that $\| v_n - p_n \|_{K_u} < \frac{\epsilon}{mG}$. 
The triangle inequality implies that $p' = \sum_{n=0}^{m}g_n p_n$ satisfies $||h-p'|| < \epsilon$. 
To put $p'$ into the form of $p$, let $2m+1$ be the maximum degree of all of the $p_n(u)$s and let $\tilde{g}_n$ 
denote the linear combinations of $g_n$ produced by grouping the terms of $g_n$ by 
the powers of $u$. The result is given by letting $E=h-p$ and writing $\dot{x} = f(x,t) + p + E$.
\end{proof}

%

\section{Example of System Not Affine in Control}\label{sec:NonAffineEx}

%

\begin{figure}[!t]
\centering
\includegraphics[width=.35\textwidth]{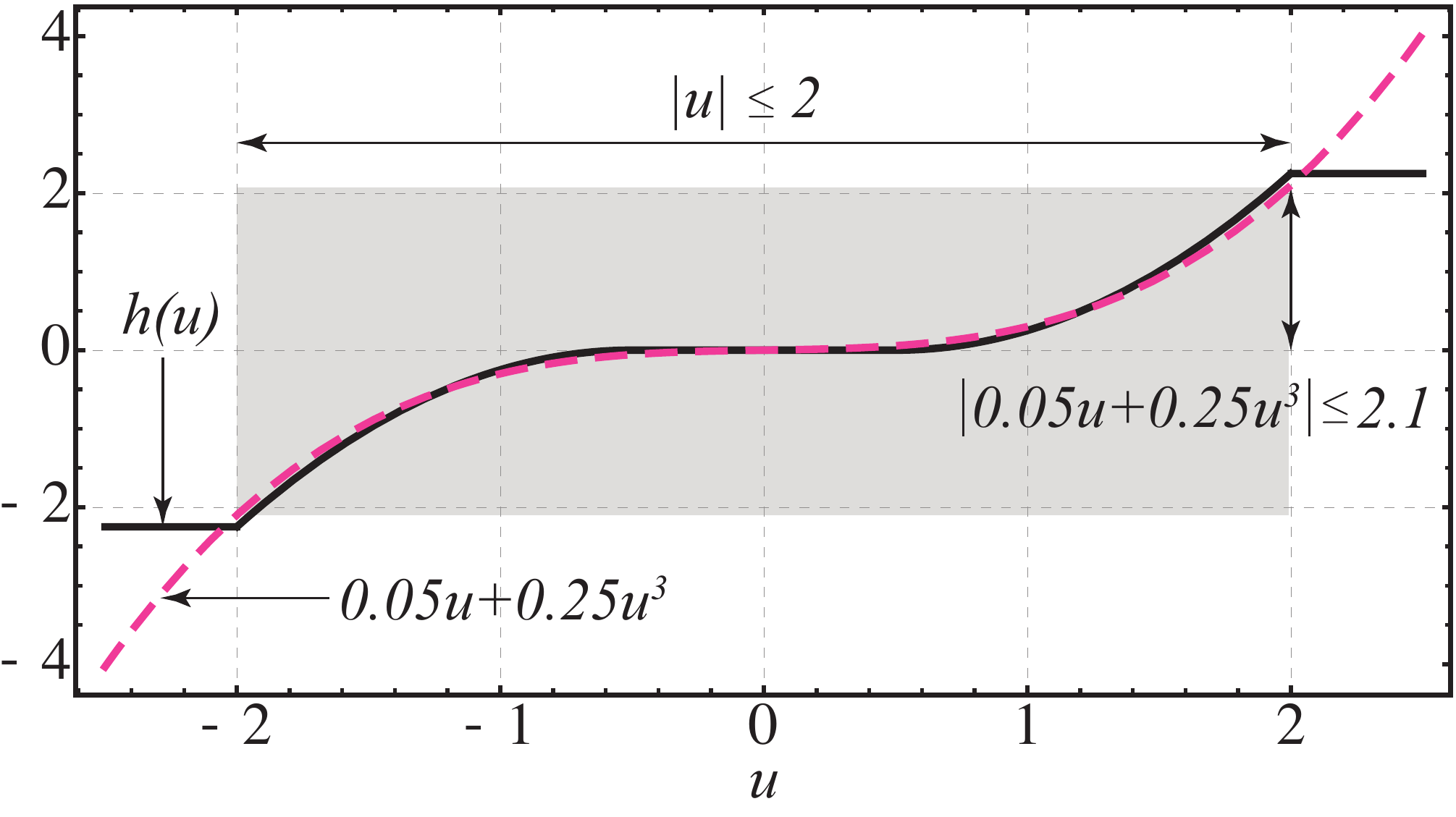}
\caption{Polynomial approximation of $h(u)$ for $|u|<2$.}
\label{fig:nonaf}
\end{figure}

Non-affine controllers, in particular non-linear controllers with dead zones, arise in a variety of practical control systems 
\cite{ref-nonaff0,ref-Recker-dead}-\cite{ref-khalil}. 
For example, a water cooling system whose flow rate is controlled by a valve with limited maximum open surface area and digital resolution-limited minimum valve opening setting. We provide examples that 
demonstrate how to develop a controller for systems in which the control effort enters 
through an odd non-linear function $h(u)$. 
In what follows we take the common approach of approximating $h(u)$ with an odd polynomial $p(u)$ \cite{ref-poly,ref-poly2}.  

\begin{example}\label{examp1}
Consider the system
\begin{equation}
	\dot{x} = f(x,t) + g(x,t)h(u), \label{nlex1}
\end{equation}
and the general approximation $h(u) \approx p(u)= a_1 u + a_3 u^3$, which produces the approximate system:
\begin{equation}
	\dot{x} = f(x,t) + g(x,t)\left (a_1 u(x,t) + a_3 u^3(x,t) \right ). \label{naf}
\end{equation}
We will design a controller of the form
\begin{equation}
	u = \left ( \alpha\omega \right )^\frac{1}{6} \cos \left ( \omega t + kV(x) \right ). \label{uex1}
\end{equation}
The reason for the choice $(\alpha\omega)^{\frac{1}{6}}$ is explained as follows. The closed loop dynamics of (\ref{naf}), (\ref{uex1}) are given by
\begin{eqnarray}
	\dot{x} &=& f(x,t) + a_1 g(x,t) \left ( \alpha\omega \right )^\frac{1}{6}\cos\left ( \omega t + kV(x) \right )  \nonumber \\
	&& + a_3 g(x,t)\sqrt{\alpha\omega} \cos^3\left ( \omega t + kV(x) \right ).
\end{eqnarray}
The $\cos()$ and $\cos^3()$ terms can be expanded as
\begin{flalign}
a_1g(x,t)
\left ( \alpha\omega \right )^\frac{1}{6} \bigg ( \cos\left ( \omega t  \right ) \cos \left (kV \right )  - \sin\left ( \omega t  \right ) \sin \left (kV \right ) \bigg )
 \nonumber
\end{flalign}
\begin{equation}
+ a_3g(x,t)0.75 \sqrt{\alpha\omega} \bigg ( \cos \left ( \omega t \right ) \cos \left ( kV \right ) - \sin \left ( \omega t \right ) \sin \left ( kV \right ) \nonumber
\end{equation}
\begin{equation}
+ 0.25\cos \left ( 3 \omega t \right ) \cos \left ( 3 kV \right ) - 0.25\sin \left ( 3 \omega t \right ) \sin \left ( 3 kV \right ) \bigg ). \nonumber
\end{equation}

Theorem \ref{thm:Sch2} implies that as $\omega \to \infty$ products containing the $a_1g(x,t)$ terms contain powers of $\omega$ of the form $1/\omega^\frac{2}{3}$ and $1/\omega^\frac{1}{3}$, which uniformly converge to zero. 
In the remaining terms, by choosing $(\alpha\omega)^{\frac{1}{6}}$, for $u^3$, we get terms with amplitudes proportional to 
$\sqrt{\alpha\omega}$ and by Theorem \ref{thm:Sch2}, are left with products in which the $\omega$ amplitude dependence has disappeared, only leaving terms of the form $\cos^2(\omega t)$, and $\sin^2(\omega t)$, which weakly converge to $1/2$, leaving us with the average system
\begin{equation}
	\dot{\bar{x}} = f(\bar{x},t) - \frac{k\alpha}{2}a^2_3g^2(\bar{x},t) \left ( \left (3/4 \right )^2 +  \left (1/4 \right )^2 \right ) \frac{\partial V(\bar{x})}{\partial \bar{x}}, \label{AvgDyn}
\end{equation}
where $\bar{x}(0)=x(0)$. Thus, to stabilize the origin, it suffices to choose $\omega, k, \alpha, V$ sufficiently large with respect to $a^2_3g^2(x,t)$ and $f(x,t)$. 

We consider the special case of System (\ref{nlex1}), where
\begin{equation}
	h(u) = \left\{\begin{array}{cc} 
	0 & |u|<0.5 \\
	\mathrm{sgn}(u)(|u|-0.5)^2 & 0.5<|u|<=2 \\
	2.25 & 2 < |u| \end{array}\right. \label{hex1}.
\end{equation}
Figure \ref{fig:exp1} shows the results of a simulation of
\begin{equation}
	\dot{x} = \frac{\cos(2t)x^2}{2}+2\cos\left ( 20 t \right )h(u), \ u = \left ( \alpha\omega \right )^\frac{1}{6} \cos \left ( \omega t + kx^2 \right )\label{ex1s1}
\end{equation}
with control parameters $\omega = 200$, $\alpha = 64/\omega$, $k=50$, and $x(0)=1.5$, where $h(u)$ is given by (\ref{hex1}) and the controller was designed using the approximation (\ref{naf}) with $a_1 = 0.05$ and $a_3 = 0.25$, which has average dynamics
\begin{equation}
	\dot{\bar{x}} = 0.5 \cos(2t) \bar{x}^2 - (5/16)k\alpha  \cos^2\left(20t\right)\bar{x}, \quad \bar{x}(0)=x(0). \label{avex1}
\end{equation}
\end{example}

Higher order systems in strict feedback form can be handled as suggested in \cite{ref-Sch-Krstic-TAC}. For example, for a second order system, one would apply the controller $u = \left ( \alpha\omega \right )^\frac{1}{6} \cos \left ( \omega t + k\left(x_1 + 2x_2\right)^2 \right )$.
\begin{figure}[!t]
\centering
\includegraphics[width=.4\textwidth]{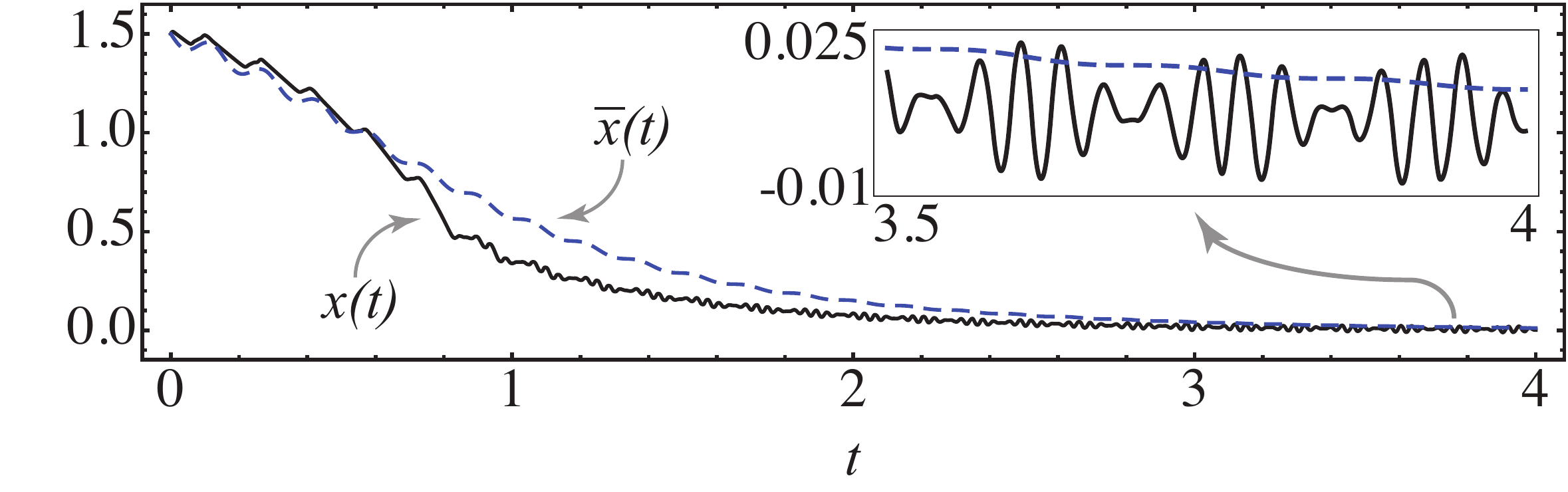}
\caption{Trajectory $x(t)$ of (\ref{ex1s1}) shown alongside the trajectory $\bar{x}(t)$ of average system (\ref{avex1}).The system behaves as expected despite our controller design being based on the approximation of the nonlinearity.}
\label{fig:exp1}
\end{figure}

\section{Robustness of Nonlinear Approximation}\label{sec:even}

%

When one designs a controller for a system of the form (\ref{NonLinearSys}), 
the function $h(x,t,u)$ may be uncertain or entirely unknown, and the approximation $\tilde{h}(x,t,u)$ may be incorrect. In this section we consider a more general class of systems with even-powered odd nonlinearities:
\begin{equation}
	\dot{x}  = f(x,t) + \sum_{i=0}^{n_o}g_{2i+1}(x,t)u^{2i+1} + \epsilon \sum_{i=1}^{n_e}g_{2i}(x,t)u^{2i}. \label{gensys}
\end{equation}
If there were no even terms and we knew that $2m+1$ was the highest power odd nonlinearity of (\ref{gensys}), then we would choose a controller based on Theorem \ref{thm1}: 
\begin{equation}
	u_m = \left(\alpha\omega\right)^{\frac{1}{2(2m+1)}}\cos \left ( \omega t + k V \right ). \label{u135u}
\end{equation} 
For unknown $n_o$, the power of our controller $2m+1$ might not equal the highest odd power of the system $2n_0+1$, and therefore the averaging analysis of Theorem \ref{thm1} breaks down when $m<n_o$ introduces terms of the form $\left(\alpha\omega\right)^{\frac{2n_o+1}{2(2m+1)}}\cos^{2n_o+1} \left ( \omega t + k V \right )$, resulting in divergent weak limits which are not independent of $\omega$, because in this case $\frac{2n_o+1}{2(2m+1)} > \frac{1}{2}$. 
Also, the averaging analysis of Theorem \ref{thm1} further breaks down since the even powers of $u$ 
introduce into the system dynamics 
positive semi-definite terms of the form $\left(\alpha\omega\right)^{\frac{2l}{2(2m+1)}}\cos^{2l} \left ( \omega t + k V \right )$ 
which grow without bound as $\omega$ grows. 
However, it turns out that we may still be able to approximate the behavior of this system. 

\begin{figure*}[!t]
\centering
\includegraphics[width=.225\textwidth]{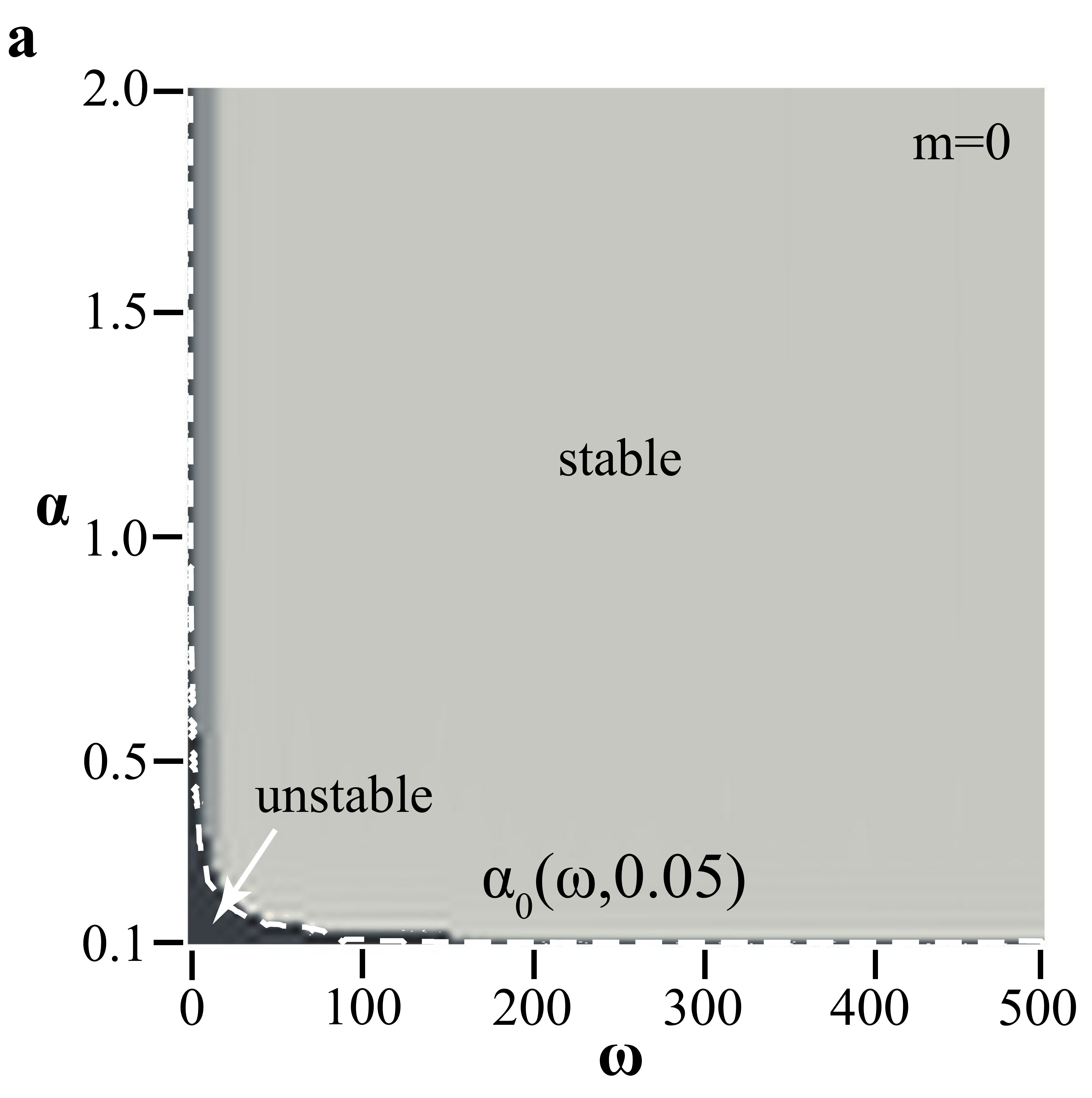}
\includegraphics[width=.225\textwidth]{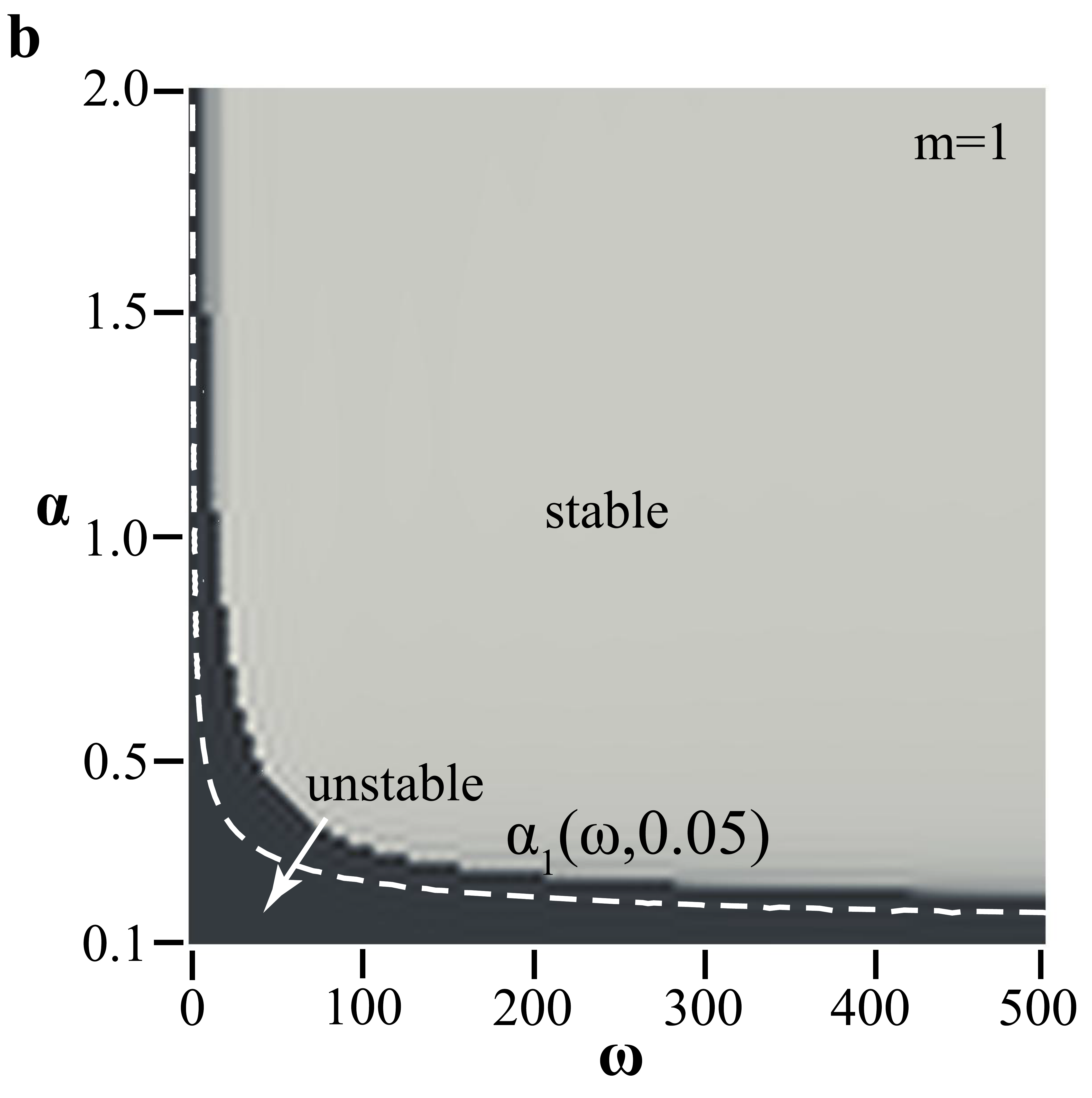}
\includegraphics[width=.258\textwidth]{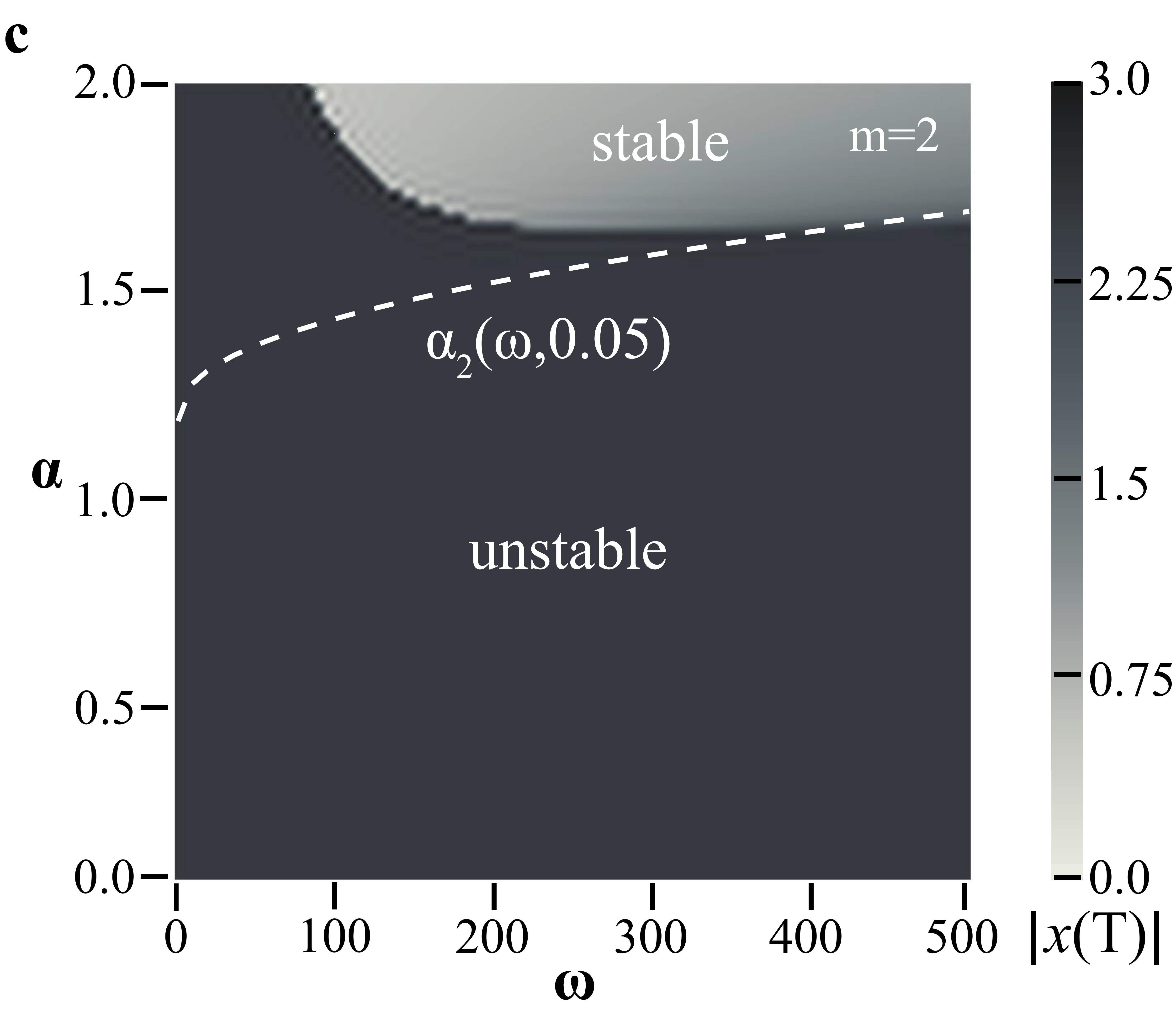}
\includegraphics[width=.263\textwidth]{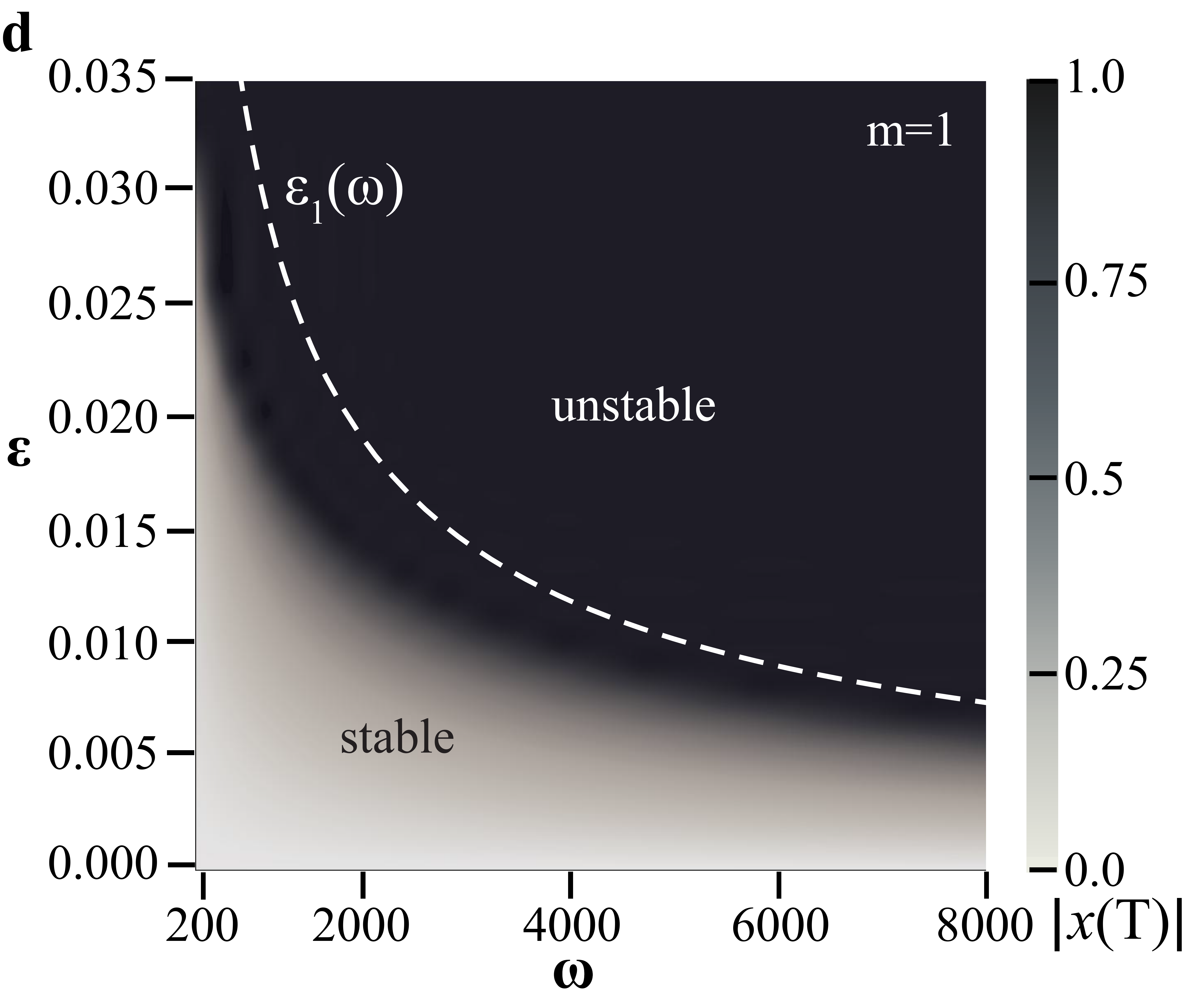}
\includegraphics[width=.225\textwidth]{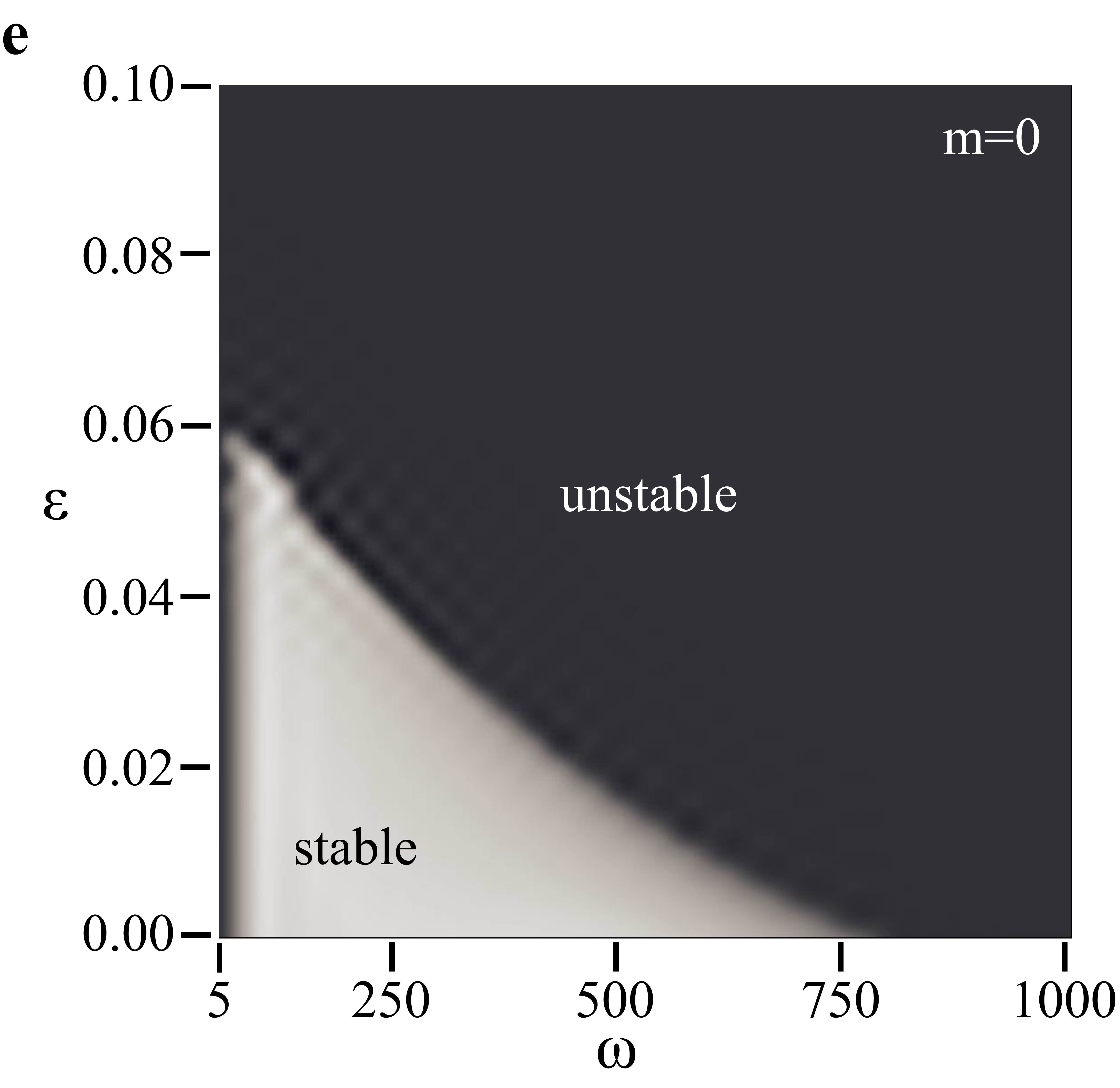}
\includegraphics[width=.225\textwidth]{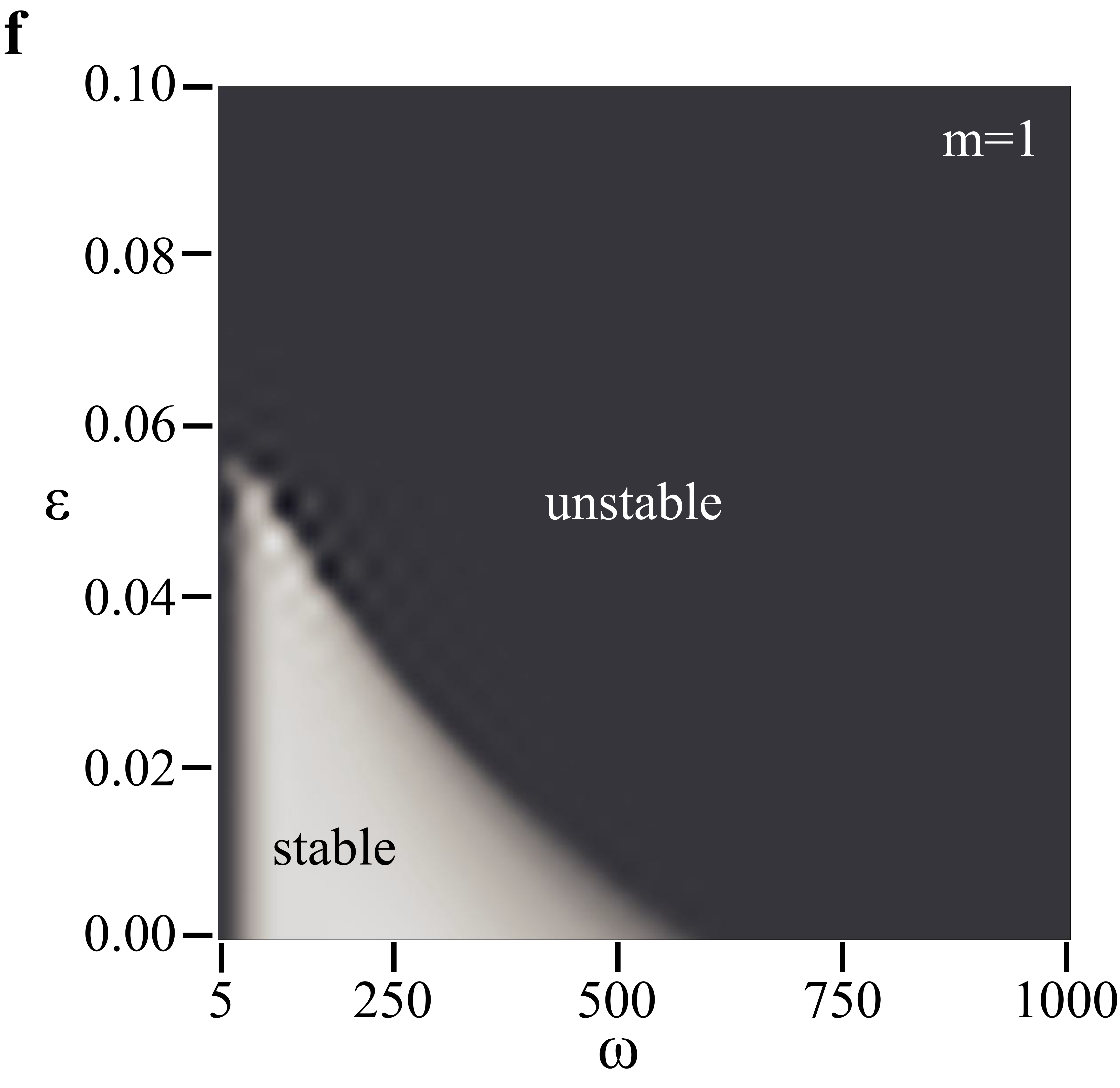}
\includegraphics[width=.258\textwidth]{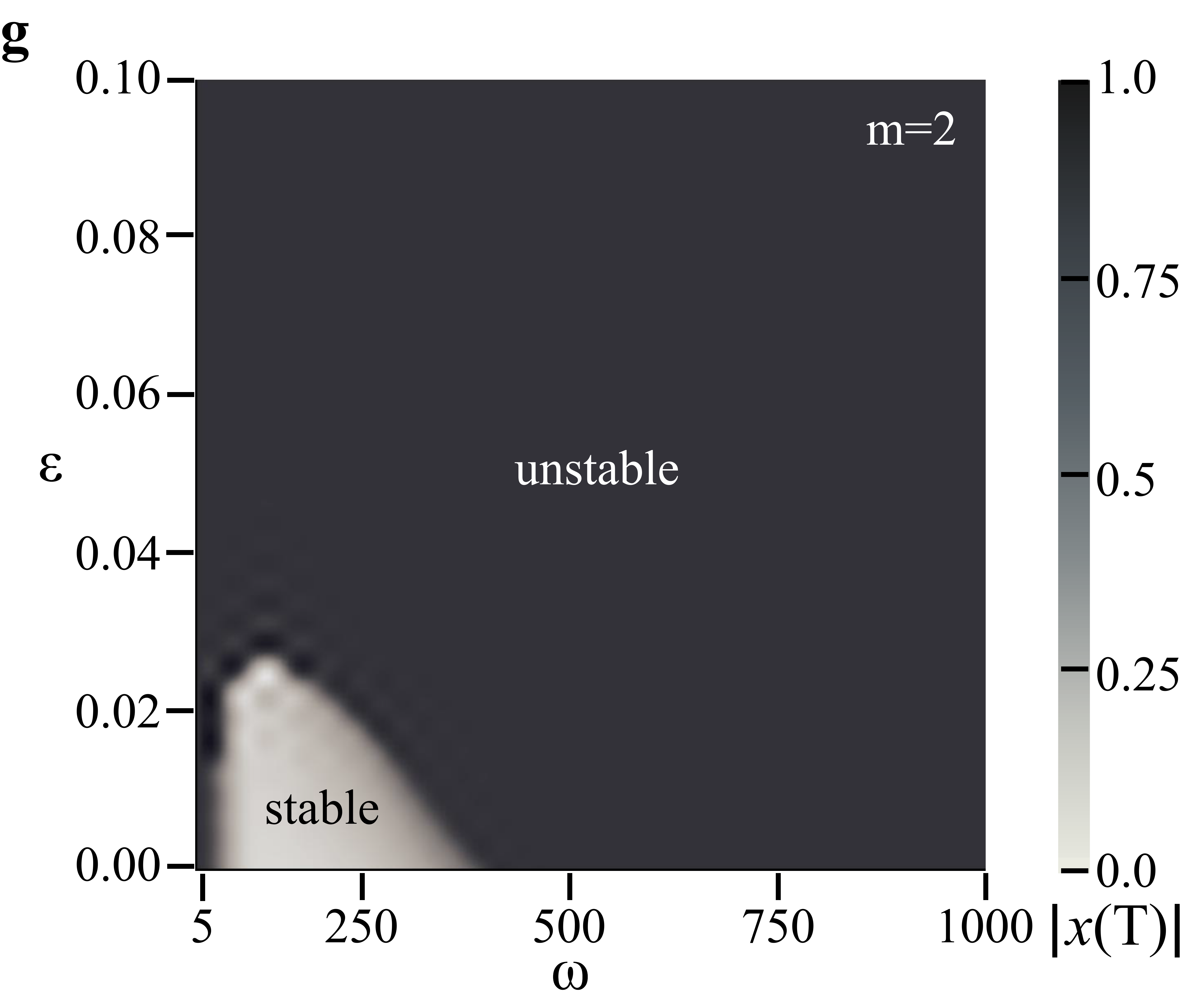}
\includegraphics[width=.263\textwidth]{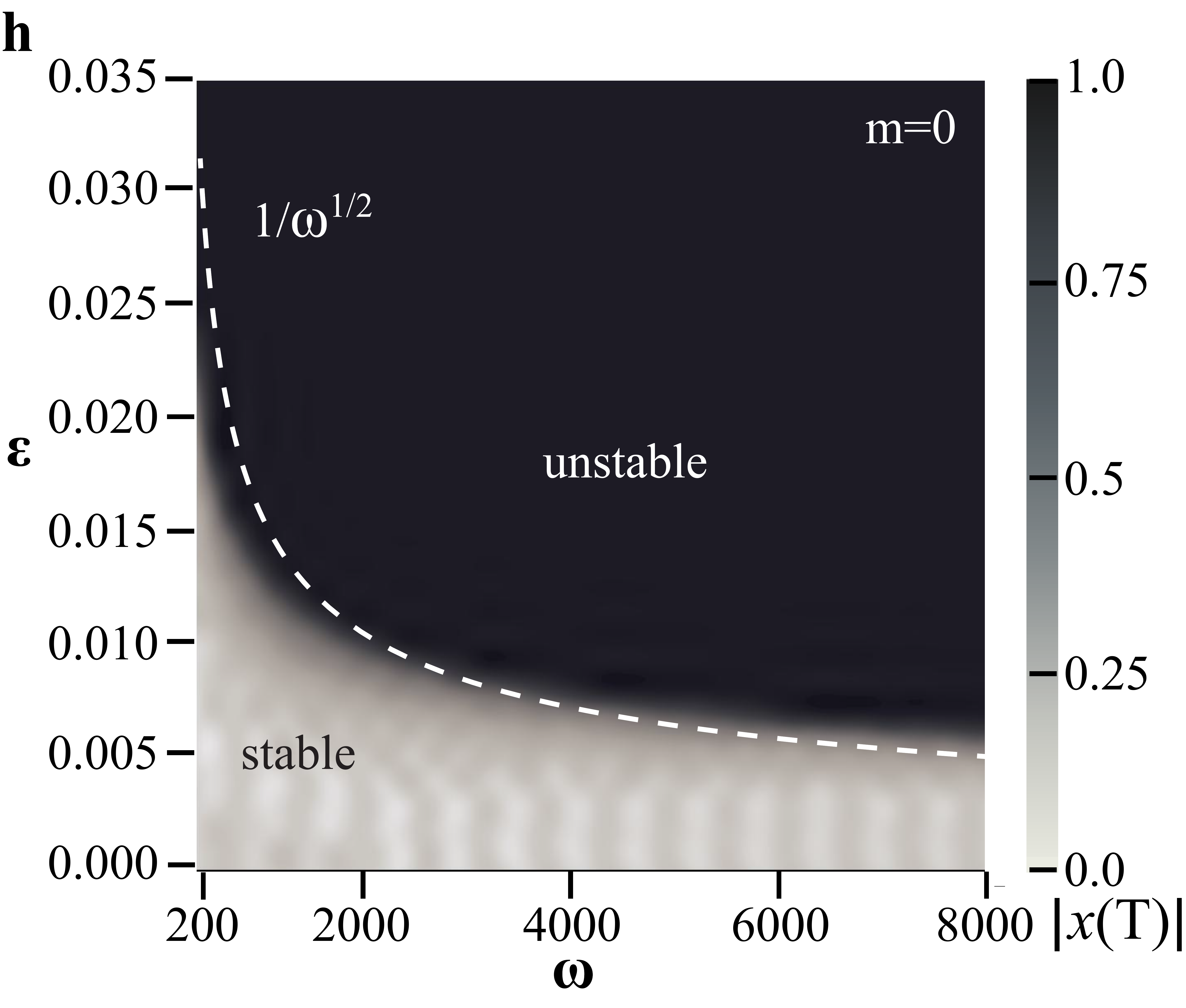}
\caption{In (a)-(c), the region of stability, as a function of $\omega$ and $\alpha$ is shown for several choices of controller, $m=0,1,2$, and $\epsilon=0.5$ for system (\ref{uu}), (\ref{uuc}), relative to the analytic prediction $\alpha_m(\omega,\epsilon,x^\star)$ based on (\ref{alphwx}). In (d) and (h), stability as a function of $\epsilon$ and $\omega$ is shown for $m=0,1$ for system (\ref{evenpow}). In (e)-(g), stability as a function of $\epsilon$ and $\omega$ is shown for $m=0,1,2$ for system (\ref{nonlfinal}).}
\label{fig:epsu135}
\end{figure*}

Integrating (\ref{gensys}) by parts, we notice that as $\omega \rightarrow \infty$ the highest power terms of (\ref{gensys}) dominate the dynamics, 
keeping only the highest order odd and even power terms, and averaging the oscillatory functions we produce the approximation (\ref{approx}), which leads us to make the following conjecture:
\begin{conjecture}\label{conj} Consider systems (\ref{gensys}), (\ref{u135u}) and:
\begin{flalign} 
	\dot{\bar{x}} &= f(\bar{x},t) + \epsilon g_{2n_e}(\bar{x},t)\left ( \alpha \omega \right )^{\frac{2n_e}{2m+1}}B_{n_e} \nonumber \\
	&- k\alpha^{\frac{2n_o+1}{2m+1}}\omega^{\left ( \frac{2n_o+1}{2m+1}-1 \right )} A_{n_0}  \frac{g_{n_o}(\bar{x},t)g^T_{n_o}(\bar{x},t)}{2^{4n_o+1}} \left (\frac{\partial V}{\partial \bar{x}}\right )^T, \nonumber \\
	& A_{n_0} = \sum_{l=0}^{n_o}\binom{2n_o+1}{l}^2, \quad B_{n_e} = \frac{1}{2^{2n_e}}\binom{2n_e}{n_e}. \label{approx}
\end{flalign} 
For any $\delta>0$, any compact set $K \subset \mathbb{R}^n$, and any $t_0, T \in \mathbb{R}_{\geq 0}$ there exists $\epsilon^\star(\delta,K,T)>0$ such that for all $\left | \epsilon \right | < \epsilon^\star$, there exists $\omega^\star$ such that for each $\omega>\omega^\star$, the trajectories $x(t)$ of (\ref{gensys}), (\ref{u135u}) and $\bar{x}(t)$ of (\ref{approx}) satisfy
\begin{equation}
	\max_{t\in \left [ t_0, t_0 + T \right ]} \left \| x(t) - \bar{x}(t) \right \| < \delta.
\end{equation}
\end{conjecture}

\begin{remark}
When $n_o = m$ and there are no even power terms, (\ref{approx}) simplifies to (\ref{nonafave}).
\end{remark}

To test Conjecture \ref{conj} we study the system
\begin{eqnarray}
	\dot{x}  &=& x + 0.1\left ( u + u^3 + u^5 \right) +  \epsilon \left ( u^2 + u^4 \right ), \label{uu} \\
	u &=& u_m = \left(\alpha\omega\right)^{\frac{1}{2(2m+1)}}\cos \left ( \omega t + k x^2 \right ). \label{uuc}
\end{eqnarray}
According to the conjecture, the closed loop dynamics should, for large $\omega$, approximate
\begin{equation}
	\dot{\bar{x}}  =  \bar{x} -\frac{2k}{100}\frac{1}{2^9}A_2\alpha^{\frac{5}{2m+1}}\omega^{\left ( \frac{5}{2m+1}-1 \right )}\bar{x} +  \epsilon B_{2} \left ( \alpha\omega \right )^{\frac{4}{2(2m+1)}}. \nonumber
\end{equation}
Therefore, the trajectory of the system should converge to the equilibrium point $x^\star$ satisfying
\begin{equation}
	 \left (\frac{2k}{100}\frac{1}{2^9}A_2\alpha^{\frac{5}{2m+1}}\omega^{\left ( \frac{5}{2m+1}-1 \right )}-1\right)x^\star = \epsilon B_{2} \left ( \alpha\omega \right )^{\frac{4}{2(2m+1)}}. \label{alphwx}
\end{equation}
Thus, for each $\epsilon>0$ and each controller power $2m+1$, we can find $\alpha_m(\omega,\epsilon,x^\star)$ which solves (\ref{alphwx}), such that for all $\alpha>\alpha_m(\omega,\epsilon,x^\star)$ the system should converge to $|x|\leq |x^\star|$. 

We confirm the estimate (\ref{alphwx}) by simulating system (\ref{uu}), (\ref{uuc}) with $k=100$, $m=1,3,5$, $\epsilon=0.05$, $\alpha \in [0.1, 2]$, $\omega \in [5, 200]$, $x(0)=1$. We let the system evolve for $T=5$ seconds and then record $\left | x(T) \right |$ (with a cutoff at 3), in order to determine whether the trajectory is converging towards the origin. We plot the results in Figure \ref{fig:epsu135} compared to the analytically predicted boundary of stability (\ref{alphwx}). For large $\left ( \alpha \omega \right )^{\frac{1}{2m}}$, the prediction is accurate.

When the highest power of the control nonlinearity is even, we have little chance of controlling the system, except in a very limited range of values of $\epsilon$ and $\omega$, because the positive semi-definite destabilizing terms dominate the dynamics and grow with $\omega$. We consider the system
\begin{equation}
	\dot{x} = 0.1u + 0.1u^3 + \epsilon u^4, \label{evenpow}
\end{equation}
with controller $u_m=\left (\alpha \omega \right )^{\frac{1}{2(2m+1)}}\cos \left ( \omega t + kx^2 \right )$, $m=0,1$. For $m=1$, Conjecture \ref{conj} implies the following estimate for a bound on $\epsilon$ for the stability of the average system:
\begin{equation}
	\epsilon < \epsilon_1(\omega) = \left ( 2k\left ( 0.1^2 \right)A_1\alpha-1 \right )/\left (\left(\alpha\omega\right)^{\frac{2}{3}}B_2 \right ).
\end{equation}
In the case $m=0$ the averaging estimates completely break down and we simply compare the dominant stabilizing $(\alpha\omega)^{\frac{3}{2}}$ and destabilizing $(\alpha\omega)^{\frac{4}{2}}$ powers and estimate a bound on $\epsilon$ of the form $\epsilon < \frac{1}{\sqrt{\alpha \omega}}$. Analytic predictions and simulation results are shown in Figure \ref{fig:epsu135} for $k=100$, $\alpha=10$.


Finally, we consider the nonlinearity from Example \ref{examp1} with even power nonlinearities, $h_\epsilon(u) = h(u) + \epsilon\left(u^2+u^4\right)$. We consider the system
\begin{equation}
	\dot{x} = x + h_\epsilon(u_m), \quad u_m = \left ( \alpha \omega \right )^{\frac{1}{2(2m+1)}}\cos\left ( \omega t + kx^2 \right), \label{nonlfinal}
\end{equation}
with various control options $m=0,1,2$ and $x(0) = 1$. The numerically calculated regions of stability are shown in Figure \ref{fig:epsu135} for $k=100$ and $\alpha=5$. The system remains stable for a small range of $\epsilon \neq 0$, and even when $\epsilon=0$, as $\omega$ is increased the system loses stability because, as $\sqrt{\alpha\omega}$ grows, $u$ leaves any fixed compact set $K_u$ and high power terms dominate.

These numerical studies show that for large $\omega$ the system is robust to the degree of odd dominating nonlinearities. In the presence of even nonlinearity, as shown in Figure \ref{fig:epsu135}(e-g), $\omega$ must be big enough, for the averaging results to approximately hold, but once $\omega$ exceeds a certain threshold it becomes destabilizing, therefore a proper operating range, probably by trial and error, must be found within the region of stability. Furthermore, if the controller is based on an underestimate of the degree of the nonlinearity and has a relatively large amplitude, 
the dithering terms will drive up the magnitude of the system's steady state oscillations (Figure \ref{fig:u135end}).

\begin{figure}[!t]
\centering
\includegraphics[width=.35\textwidth]{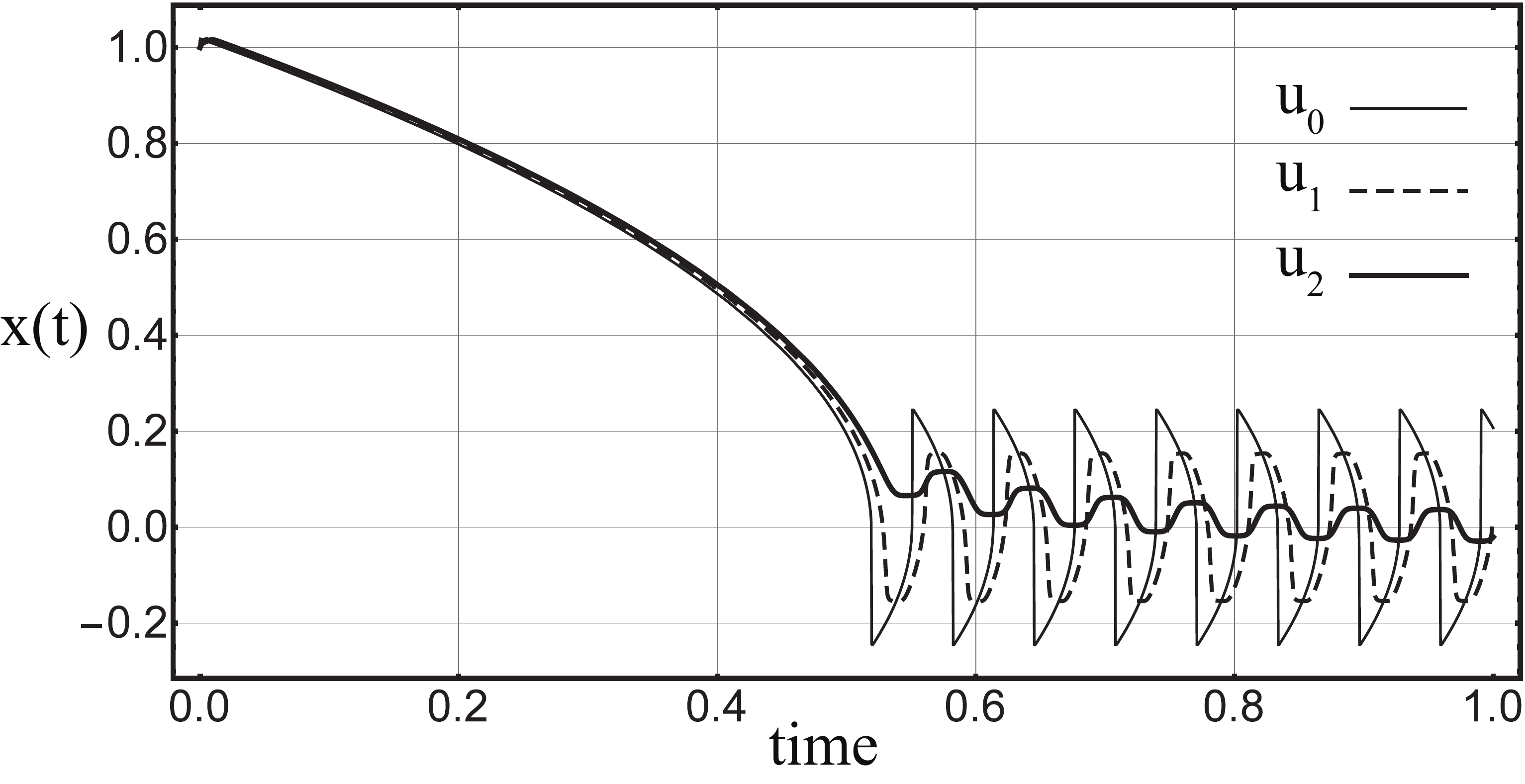}
\caption{A simulation of system (\ref{uu}), (\ref{uuc}), comparing the use of controllers with $m=0,1,2$.}
\label{fig:u135end}
\end{figure}


%

\section{Conclusions}\label{sec:con}

%

This work extends the application of ESC to systems not affine in control by approximating the non-linearities with odd functions and designing a controller for the approximate system. 
If the control input function is odd, then we provide analytic average results guaranteeing the satiability of the system. 
For the case when the input function is not odd with respect to $u$, we conjecture a sufficient condition for our controller to stabilize the system. 
Our result produces an average system nearly identical to that in \cite{ref-Sch-NewES} where the control entered the dynamics linearly as:
$\dot{x} = f(x,t) +  \sqrt{K_g}g_m(x,t)u$.
%








\begin{thebibliography}{99}

%


\bibitem{ref-lebES} M. Leblanc, ``Sur l'electri"cation des chemins de fer au moyen de courants alternatifs de frequence elevee," {\it Revue Generale de l'Electricite}, 1922.

\bibitem{ref-ESOst} I. I. Ostrovskii, ``Extremum regulation," {\it Automatic and Remote Control}, vol. 18, pp. 900-907, 1957.

\bibitem{ref-ESMor} S. Morosanov, ``Method of extremum control," {\it Automatic and Remote Control}, vol. 18, pp. 1077-1092, 1957.

\bibitem{ref-ESOb} V. K. Obabkovm ``Theory of multichannel extremal control systems with sinusoidal probe signals," {\it Automation and Remote Control}, vol. 28, pp. 48-54, 1967.

\bibitem{ref-Meerk2} S. Meerkov, ``Asymptotic methods for investigating a class of forced states in extremal systems," {\it Automation and Remote Control}, vol.28, pp. 1916-1920, 1967.

\bibitem{ref-ESVol} V. M. Volosov, ``Averaging in systems of ordinary differential equations," Uspekhi matem. nauk, vol.17, No. 6, pp. 108, 1962.

\bibitem{ref-krstic} M. Krsti\'c and H. Wang, ``Stability of extremum seeking feedback for general dynamic systems.'' {\it Automatica, } vol. 36, pp. 595-601, 2000.

\bibitem{ref-tan06} Y.~Tan, D.~Ne\v{s}i\'{c}, and I.~Mareels, ``On non-local stability properties of extremum seeking control,'' {\it Automatica}, vol. 42, pp. 889-903, 2006.


\bibitem{ref-guay1} M. Guay and T. Zhang, ``Adaptive extremum seeking control of nonlinear dynamics systems with parametric uncertainties," {\it Automatica}, vol. 39, pp. 1283-1293, 2003.

\bibitem{ref-ES-Paek} Y. Li, M. Rotea, G. Chiu, L. Mongeau, I. Paek, ``Extremum seeking control of a tunable thermoacoustic cooler," {\it IEEE Transactions on Automatic Control}, vol. 13, pp. 527-536, 2005.

\bibitem{ref-ES-Manzie} D. Nesic, T. Nguyen, Y. Tan, C. Manzie, ``A non-gradient approach to global extremum seeking: an adaptation of the Shubert algorithm," {\it Automatica}, vol. 49, pp. 809-815, 2013.

\bibitem{ref-guay1} M. Guay and T. Zhang, ``Adaptive extremum seeking control of nonlinear dynamics systems with parametric uncertainties," {\it Automatica}, vol. 39, pp. 1283-1293, 2003.

\bibitem{ref-Krstic-power} A. Ghaffari, M. Krstic, and S. Seshagiri, ``Power optimization for photovoltaic micro-converters using multivariable Newton-based extremum seeking," {\it IEEE Transactions on Control Systems Technology}, vol. 22, pp. 2141-2149, 2014.

\bibitem{ref-Krstic-stoch} S.-J. Liu and M. Krstic, ``Newton-based stochastic extremum seeking," {\it Automatica}, vol. 50, pp. 952-961, 2014.

\bibitem{ref-guay2} M. Guay, ``A time-varying extremum-seeking control approach for discrete-time systems," {\it Journal of Process Control}, vol. 24, pp. 98-112, 2014.

\bibitem{ref-Krstic-unst} C. Zhang, A. Siranosian, and M. Krstic, ``Extremum seeking for moderately unstable systems and for autonomous vehicle target tracking without position measurements,Ó {\it Automatica}, vol. 43, pp. 1832Ð1839, 2007.

\bibitem{ref-kurz-jar-1} J. Kurzweil, J. Jarnik, ``Limit processes in ordinary differential equations," {\it Journal of Applied Mathematics and Physics}, vol. 38, pp. 241-256, 1987.

\bibitem{ref-suss-liu} H. J. Sussmann, W. Liu, ``Limits of highly oscillatory controls and approximation of general paths by admissible trajectories," Proc. 30th IEEE CDC, Brighton, UK, 1991.


\bibitem{ref-mor-aeyel-00} L. Moreau and D. Aeyels, ``Practical stability and stabilization.'' {\it IEEE Transactions on Automatic Control}, vol. 45, pp. 1554-1558, 2000.

\bibitem{ref-durr-stan-john-11} H. D\"urr, M. Stankovi\'c, C. Ebenbauer, K. Johansson, ``Lie bracket approximation of extremum seeking systems," {\it Automatica}, vol. 49, pp. 1538 - 1552, 2013.

\bibitem{ref-ES-REV} W. H. Moase, C. Manzie, D. Nesic, and I.M.Y. Mareels, ``Extremum seeking from 1922 to 2010," in 29th {\it Chinese Control Conference}, 14, 2010.


\bibitem{ref-Sch-Krstic-TAC} A. Scheinker, M. Krstic, ``Maximum-seeking for CLFs: Universal semiglobally stabilizing feedback under unknown control directions," {\it IEEE Transactions on Automatic Control}, vol. 58, pp. 1107-1122, 2013.

\bibitem{ref-kapitza} P.L. Kapitza, ``Dynamic stability of a pendulum when its point of suspension vibrates," Soviet Phys. JETP {\bf 21}, 588592, 1951.

\bibitem{ref-Meerk} S. Meerkov, ``Principle of Vibrational Control: Theory and Applications," {\it IEEE Transactions on Automatic Control}, vol. 14, pp. 755-762, 1980.

\bibitem{ref-HVCM} A. Scheinker, M. Bland, M. Krstic, J. Audia, ``Rise-time optimization of high voltage converter modulator rise-time," {\it IEEE Transactions on Control Systems Technology}, vol. 22, pp. 34-43, 2013.

\bibitem{ref-pend} S. Michalowsky, C. Ebenbauer, ``Swinging up the Stephenson-Kapitza pendulum," Proc. 52nd IEEE CDC, Florence, Italy, 2013.

\bibitem{ref-manifold} H. B. Durr, M. S. Stankovic, K. H. Johansson, and C. Ebenbauer, ``Extremum seeking on submanifolds in the Euclidian space," {\it Automatica}, vol.50, pp. 2591-2596, 2014. 

\bibitem{ref-Sch-Krstic-nonC2} A. Scheinker, M. Krstic, ``Non-$C^2$ Lie bracket averaging for non-smooth extremum seekers," {\it ASME Journal of Dynamic Systems, Measurement, and Control}, vol. 136, 2013.

\bibitem{ref-Sch-NewES} A. Scheinker, M. Krstic, ``Extremum seeking with bounded update rates," {\it Systems \& Control Letters}, vol. 63, pp. 25-31, 2014.

\bibitem{ref-Sch-NIM} A. Scheinker, S. Baily, D. Young, J. Kolski, M. Prokop, ``In-hardware demonstration of model-independent adaptive tuning of noisy systems with arbitrary phase drift," {\it Nuclear Instruments and Methods in Physics Research A}, vol. 756. pp.30-38, 2014.

\bibitem{ref-Sch-Sch} A. Scheinker, D. Scheinker ``Bounded extremum seeking with discontinuous dithers," to appear, {\it Automatica}, vol. 52, Issue 8, August 2016.

\bibitem{ref-nonaff0} T. Hu and Z. Lin, ``Control systems with actuator saturation: analysis and design," Springer Science \& Business Media, 2001.

\bibitem{Rudin} Rudin, Walter. Real and complex analysis (3rd). New York: McGraw-Hill Inc, 1986.

\bibitem{ref-Recker-dead} Recker, D. A., et al. ``Adaptive nonlinear control of systems containing a deadzone." Decision and Control, 1991., Proceedings of the 30th IEEE Conference on. IEEE, 1991.

\bibitem{ref-Selmic-dead} Selmic, Rastko R., and Frank L. Lewis. "Deadzone compensation in motion control systems using neural networks." Automatic Control, IEEE Transactions on 45.4, 602-613, 2000

\bibitem{ref-Nordin-nonlinear} Nordin, Mattias, and Per-Olof Gutman. "Controlling mechanical systems with backlash, a survey." {\it Automatica} vol. 38, pp. 1633-1649,  2002.

\bibitem{ref-khalil} H. Khalil, {\it Nonlinear Systems}, Prentice Hall, Upper Saddle River, NJ, 2002.

\bibitem{ref-poly} X. Y. Gu and C. Shao. "Robust adaptive control of time-varying linear plants using polynomial approximation." {\it Control Theory and Applications}, IEE Proceedings D. Vol. 140. No. 2. IET, 1993.

\bibitem{ref-poly2} N. Kryloff and N. Bogoliuboff. {\it Introduction to Nonlinear Mechanics}, Princeton University Press, NJ, 1943.

\end{thebibliography}
\end{document}